\documentclass[preprint, review,12pt]{article}

\usepackage{amssymb}
\usepackage[leqno]{amsmath}
\usepackage{graphicx}

\usepackage{cases}
\usepackage{amsthm}
\usepackage[german,english]{babel}

\usepackage{epstopdf}

\makeatletter
\def\fnum@figure{Fig.\thefigure}
\makeatother
\newtheorem{theorem}{Theorem}
\theoremstyle{definition}
\newtheorem{remark}{Remark}
\numberwithin{remark}{section}
\numberwithin{theorem}{section}
\numberwithin{equation}{section}
\begin{document}
\title{A note of ``pointwise estimates of the SDFEM for convection--diffusion problems with characteristic layers''}

\author{%
Jin Zhang\thanks{Email: JinZhangalex@hotmail.com}
\footnote{Address:School of Science, Xi'an Jiaotong University,
Xi'an, 710049, China}}
\maketitle
\begin{abstract}
We propose some useful estimates for the pointwise error estimates of the streamline
diffusion finite element method (SDFEM) on Shishkin meshes, when SDFEM is applied for problems of characteristic layers.
\end{abstract}
\section{Problem}
We consider the singularly perturbed boundary value problem
\begin{equation}\label{problem}
 \begin{array}{rl}
-\varepsilon\Delta u+bu_{x}+cu=f & \text{in $\Omega=(0,1)^{2}$},\\
 u=0 & \text{on $\partial\Omega$}
 \end{array}
\end{equation}
where $b$, $c>0$ are constants and $b\ge\beta$ on $\overline{\Omega}$ with a positive constant $\beta$. It is assumed that $f$ is sufficiently smooth.
Here $0<\varepsilon\ll 1$  is a small perturbation
parameter whose presence gives rise to an exponential layer of width
$O(\varepsilon)$ near the outflow boundary at $x=1$ and to two characteristic (or parabolic) layers of width $O(\sqrt{\varepsilon})$ near the characteristic boundaries at $y=0$
and $y=1$.
\section{The SDFEM on Shishkin meshes}
In this Section we describe our mesh, our finite element method and the assumptions of our analysis.
\subsection{The regularity result}
As mentioned before the solution $u$ of \eqref{problem} possesses an exponential layer at $x=1$ and two characteristic layers at $y=0$ and $y=1$. For our later analysis we shall suppose
that $u$ can be split into a regular solution component and
various layer parts:
\newtheorem{assumption}{Assumption}[section]
\begin{assumption}
The solution $u$ of \eqref{problem} can be decomposed as
\begin{subequations}\label{Assumption 1}
\begin{equation}\label{eq:(2.1a)}
u=S+E_{1}+E_{2}+E_{12},
\end{equation}
where for all $\boldsymbol{x}=(x,y)\in\bar{\Omega}$ and for $0\le i+j \le 3$, the regular part satisfies
\begin{equation}\label{eq:(2.1b)}
\left|\frac{\partial^{i+j}S}{\partial x^{i}\partial y^{j}}(x,y) \right|\le C,
\end{equation}
while for the layer terms and $0\le i+j \le 3$, the following bounds hold true\upshape{:}
\begin{equation}\label{eq:(2.1c)}
\left|\frac{\partial^{i+j}E_{1}}{\partial x^{i}\partial y^{j}}(x,y) \right|\le C\varepsilon^{-i}e^{-\beta(1-x)/\varepsilon},
\end{equation}
\begin{equation}\label{eq:(2.1d)}
\left|\frac{\partial^{i+j}E_{2}}{\partial x^{i}\partial y^{j}}(x,y) \right|\le C\varepsilon^{-j/2}(e^{-y/\sqrt{\varepsilon}}+e^{-(1-y)/\sqrt{\varepsilon}}),
\end{equation}
and
\begin{equation}\label{eq:(2.1e)}
\quad\quad\quad
\left|\frac{\partial^{i+j}E_{12}}{\partial x^{i}\partial y^{j}}(x,y) \right|\le C\varepsilon^{-(i+j/2)}e^{-\beta(1-x)/\varepsilon}
(e^{-y/\sqrt{\varepsilon}}+e^{-(1-y)/\sqrt{\varepsilon}}).
\end{equation}
\end{subequations}
\end{assumption}
\setcounter{equation}{1}

\par
For constant coefficients Kellogg and Stynes \cite{Kell1Styn2:2005-Corner,Kell1Styn2:2007-Sharpened} give sufficient compatibility conditions on the data that ensure
the existence of \eqref{eq:(2.1a)}--\eqref{eq:(2.1e)}.

\subsection{Shishkin meshes}
 When discretizing \eqref{problem}, we use a piecewise uniform mesh --- a so-called
 \textit{Shishkin} mesh ---with $N$ mesh intervals in both $x-$ and $y-$ direction which
 condenses in the layer regions. For this purpose we define the two mesh transition parameters
\begin{equation*}
\lambda_{x}:=\min\left\{ \frac{1}{2},\rho\frac{\varepsilon}{\beta}\ln N \right\} \quad \mbox{and} \quad
\lambda_{y}:=\min\left\{
\frac{1}{4},\rho\sqrt{\varepsilon}\ln N \right\}.
\end{equation*}
In this paper, we define $\rho=2.5$.
\begin{assumption}\label{Assumption 2}
We assume in our analysis that $\varepsilon\le N^{-1}$. Furthermore we assume that
$\lambda_{x}=\rho\varepsilon\beta^{-1}\ln N$ and $\lambda_{y}=\rho\sqrt{\varepsilon}\ln N$ as otherwise $N^{-1}$ is exponentially small compared with $\varepsilon$.
\end{assumption}
\par
The domain $\Omega$ is dissected into four(six) parts as $\Omega=\Omega_{s}\cup\Omega_{1}\cup\Omega_{2}\cup\Omega_{12}$, where
\begin{align*}
&\Omega_{s}:=\left[0,1-\lambda_{x}\right]\times\left[\lambda_{y},1-\lambda_{y}\right],&&
\Omega_{2}:=\left[0,1-\lambda_{x}\right]\times\left(\left[0,\lambda_{y}\right]
\cup\left[1-\lambda_{y},1\right]
\right),\\
&\Omega_{1}:=\left[ 1-\lambda_{x},1 \right]\times\left[\lambda_{y},1-\lambda_{y}\right],&&
\Omega_{12}:=\left[ 1-\lambda_{x},1 \right]\times\left(\left[0,\lambda_{y}\right]
\cup\left[1-\lambda_{y},1\right]
\right).
\end{align*}
\begin{remark}
The mesh transition parameters have been chosen such that the boundary layer function $E$ which can be any of $E_{1},E_{2}$ and $E_{12}$ satisfies $\left| E\right|\le CN^{-\rho}$ on $\Omega_{s}$.
\end{remark}

\par
We introduce the set of mesh points $\left\{ (x_{i},y_{j})\in\Omega:i,\,j=0,\,\cdots,\,N  \right\}$ defined by
\begin{equation*}
x_{i}=
\left\{
\begin{array}{ll}
2i(1-\lambda_{x})/N ,&\text{for $i=0,\,\cdots,\,N/2$},\nonumber \\
1-2(N-i)\lambda_{x}/N, &\text{for $i=N/2+1,\,\cdots,\,N$}\nonumber
\end{array}
\right.
\end{equation*}
and
\begin{equation*}
y_{j}=
\left\{
\begin{array}{ll}
3j\lambda_{y}/N ,&\text{for $j=0,\,\cdots,\,N/3$}, \nonumber \\
(3j/N-1)-3(2j-N)\lambda_{y}/N ,&\text{for $j=N/3+1,\,\cdots,\,2N/3$},\nonumber\\
1-3(N-j)\lambda_{y}/N, &\text{for $j=2N/3+1,\,\cdots,\,N$}.\nonumber
\end{array}
\right.
\end{equation*}
By drawing lines through these mesh points parallel to the $x$-axis and $y$-axis the domain $\Omega$ is partitioned into rectangles. This triangulation is denoted by $\Omega^{N}$. If $D$ is a mesh subdomain of $\Omega$, we write $D^{N}$ for the triangulation of $D$. The mesh sizes $h_{x,\tau}=x_{i}-x_{i-1}$ and $h_{y,\tau}=y_{j}-y_{j-1}$ satisfy
\begin{numcases}{h_{x,\tau}=}
H_{x}:=\frac{1-\lambda_{x}}{N/2}, &\text{for $i=1,\,\cdots,\,N/2$}, \nonumber\\
h_{x}:=\frac{\lambda_{x}}{N/2}, &\text{for $i=N/2+1,\,\cdots,\,N$}, \nonumber
\end{numcases}
\begin{numcases}{h_{y,\tau}=}
H_{y}:=\frac{1-2\lambda_{y}}{N/3}, &\text{for $j=N/3+1,\,\cdots,\,2N/3$},\nonumber \\
h_{y}:=\frac{\lambda_{y}}{N/3}, &\text{otherwise}\nonumber
\end{numcases}
and
\begin{align*}
N^{-1}\le H_{x},H_{y} \le 3N^{-1} ,\\
C_{1}\varepsilon N^{-1}\ln N \le h_{x}\le C_{2}\varepsilon N^{-1}\ln N,\\
C_{1}\sqrt{\varepsilon} N^{-1}\ln N \le h_{y}\le C_{2}\sqrt{\varepsilon} N^{-1}\ln N.
\end{align*}
 The above properties are essential when inverse inequalities
are applied in our later analysis.
\begin{figure}
\centering
\includegraphics[width=0.6\textwidth]{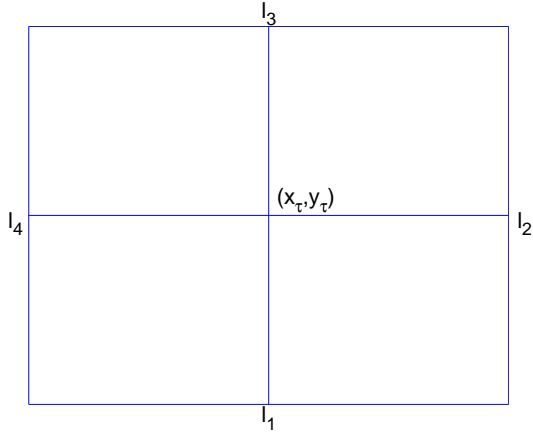}
\caption{\footnotesize{Geometry of the element $\tau$ }}
\label{fig:2}
\end{figure}
\par
For the mesh elements we shall use some notations: $\tau_{ij}=[x_{i-1},x_{i}]\times [y_{j-1},y_{j}]$ for a specific element, $\tau$ for a generic mesh rectangle (see Fig.\ref{fig:2}) and
\begin{equation*}
x_{\tau}=(x_{i-1}+x_{i})/2,\quad y_{\tau}=(y_{j-1}+y_{j})/2
\quad\mbox{if}\quad\tau=[x_{i-1},x_{i}]\times [y_{j-1},y_{j}].
\end{equation*}

\subsection{The streamline diffusion finite element method}

The weak formulation of the
 problem \eqref{problem} is: Find $u\in H^{1}_{0}(\Omega)$ such that
\begin{equation}\label{weak formulation}
\varepsilon (\nabla u,\nabla v)+(bu_{x}+cu,v)=(f,v),\,\forall v\in H^{1}_{0}(\Omega).
\end{equation}
Note that the variational formulation \eqref{weak formulation} has a unique solution by means of
the Lax-Milgram Lemma.
\par
On the above Shishkin mesh we define a finite element space
\begin{equation*}
V^{N}:=\{v^{N}\in C(\bar{\Omega}):v^{N}|_{\partial\Omega}=0
 \text{ and $v^{N}|_{\tau}$ is bilinear, }
  \forall\tau\in\Omega^{N} \}.
\end{equation*}
Then we can state the standard Galerkin discretisation of \eqref{weak formulation} is: Find $U\in V^{N}$ such that
\begin{equation}\label{Galerkin}
\varepsilon (\nabla U,\nabla v^{N})+(bU_{x}+cU,v^{N})=(f,v^{N}),\,\forall v^{N}\in V^{N}.
\end{equation}

The SDFEM adds weighted residuals to the standard Galerkin finite
element method :
Find $U\in V^{N}$ such that
\begin{equation}\label{SDFEM}
B(U,v^{N})=(f,v^{N}+\delta bv^{N}_{x}),\,\forall v^{N}\in V^{N},
\end{equation}
where
\begin{equation}\label{bilinear form}
B(U,v^{N}):=\varepsilon(\nabla U,\nabla v^{N})+(bU_{x}+cU,v^{N})
+(bU_{x}+cU,\delta bv^{N}_{x}).
\end{equation}
The term $(-\varepsilon\Delta U,\delta bv^{N}_{x})$ is
neglected in our case. $\delta=\delta(\boldsymbol{x})$ is a user-chosen parameter (see \cite{Johnson:1987-Numerical,Roo1Sty2Tob3:2008-Robust}).
In this paper, we set
\begin{numcases}{\delta(\boldsymbol{x}):=}
C^{\ast}N^{-1},&\text{if $\boldsymbol{x}\in\Omega_{s}\cup\Omega_{2}$},\nonumber\\
0,&\text{otherwise}\nonumber
\end{numcases}
where $C^{\ast}$ is $O(1)$ and $N$ satisfies the relation $0<C^{\ast}N^{-1}\le 1/c$ for $\boldsymbol{x}\in\Omega_{s}\cup\Omega_{2}$ (see \cite[\S III 3.2.1]{Roo1Sty2Tob3:2008-Robust}).
Finally, we define a special energy norm associated with $B(\cdot,\cdot)$:
\begin{equation}\label{energy norm}
\vert\vert\vert U \vert\vert\vert^{2}:=
((\varepsilon+b^{2}\delta)U_{x},U_{x})+\varepsilon(U_{y},U_{y})+(cU,U).
\end{equation}

\par
For any subdomain $D$ of $\Omega$, let $B_{D}(\cdot,\cdot),\,(\cdot,\cdot)_{D}$ and $\vert\vert\vert\cdot\vert\vert\vert_{D}$ mean that the integrations in \eqref{bilinear form} and \eqref{energy norm} are restricted to $D$. We denote by $\Vert\cdot\Vert_{D}$ the $L^{2}$ norm in $L^{2}(D)$, i.e.,
\begin{equation*}
\Vert v\Vert^{2}_{D}=(v,v)_{D} \text{  for all } v\in L^{2}(D).
\end{equation*}
If $D=\Omega$ then we drop $\Omega$ from the notation.

%%%%%%%%%%%%%%%%%%%%%%%%%%%%%%%%%%%%%%%%%%%%%%%%%%%%%%%%%%%%%%%%%%%%%%%%%%%%%%%%%%%%%%%%%%%%%%%%%%%%%%%%%%%%
%
%
%%%%%%%%%%%%%%%%%%%%%%%%%%%%%%%%%%%%%%%%%%%%%%%%%%%%%%%%%%%%%%%%%%%%%%%%%%%%%%%%%%%%%%%%%%%%%%%%%%%%%%%%%%%%%
\section{Interpolation error estimates}
We start our analysis by quoting some previous results. In the following analysis, we shall frequently use the bilinear interpolation $g^{I}$ of a given function $g$.
\newtheorem{lemma}{Lemma}[section]
\begin{lemma}\label{interpolation}
Assume that $u$ satisfies Assumption $2.1$. Then  on our Shishkin mesh
\begin{numcases}{\Vert u-u^{I} \Vert_{L^{\infty}(\tau)}\le}
CN^{-2},&\text{if $\tau\in\Omega_{s}$,}\nonumber\\
CN^{-2}\ln^{2}N, &\text{otherwise.}\nonumber
\end{numcases}
\end{lemma}
\begin{proof}
See \cite[Theorem 3]{Fra1Lin2Roo3:2008-Superconvergence}.
\end{proof}

\begin{lemma}\label{Lin identities}
For any function $g\in C^{3}(\tau)$ and any $w\in V^{N}$, we have the identities
\begin{subequations}
\begin{align}
\int_{\tau}(g-g^{I})_{x}w_{x}\mathrm{d}x\mathrm{d}y=&
\int_{\tau}g_{xyy}J_{\tau}(y) \left(
w_{x}-\frac{2}{3}(y-y_{\tau})w_{xy}
\right)\mathrm{d}x\mathrm{d}y,\\
\int_{\tau}(g-g^{I})_{y}w_{y}\mathrm{d}x\mathrm{d}y=&
\int_{\tau}g_{xxy}F_{\tau}(x) \left(
w_{y}-\frac{2}{3}(x-x_{\tau})w_{xy}
\right)\mathrm{d}x\mathrm{d}y
\end{align}
and
\begin{equation}
\int_{\tau}(g-g^{I})_{x}w\mathrm{d}x\mathrm{d}y=
\int_{\tau}R(g,w)\mathrm{d}x\mathrm{d}y+
\frac{h^{2}_{x,\tau}}{12}\left(\int_{l_{2}}-\int_{l_{4}}\right)
g_{xx}w\mathrm{d}y
\end{equation}
\end{subequations}
where
\begin{equation*}
F_{\tau}(x)=\frac{1}{2}\left(
(x-x_{\tau})^{2}-h^{2}_{x,\tau}/4
 \right)
 \quad \mbox{and} \quad
J_{\tau}(y)=\frac{1}{2}\left(
(y-y_{\tau})^{2}-h^{2}_{y,\tau}/4
 \right)
\end{equation*}
and
\begin{align*}
R(g,w)=&\frac{1}{3}F_{\tau}(x)(x-x_{\tau})g_{xxx}w_{x}-\frac{h^{2}_{x,\tau}}{12}g_{xxx}w
+J_{\tau}(y)g_{xyy}\\
&\left[
w-(x-x_{\tau})w_{x}-\frac{2}{3}(y-y_{\tau})w_{y}+\frac{2}{3}(x-x_{\tau})
(y-y_{\tau})w_{xy}
\right].
\end{align*}
\end{lemma}
\begin{proof}
See \cite{Lin:1991-rectangle,Lin1Yan2:1996-Construction} or \cite[the Appendix]{Zhang:2003-Finite} for details.
\end{proof}

\begin{lemma}\label{uI U energy estimate}
Let $U$ be the solution of \eqref{SDFEM} on our Shishkin mesh. Then
\begin{equation*}
\vert\vert\vert u^{I}-U \vert\vert\vert
\le
C(N^{-2}\ln^{2}N+\varepsilon^{1/4}N^{-1}\ln^{1/2}N)
\end{equation*}
\end{lemma}
\begin{proof}
See \cite[Theorem 5]{Fra1Lin2Roo3:2008-Superconvergence}.
\end{proof}
Throughout the remaining analysis we shall make frequent use of the following inverse
estimates. Let $\chi$ be a polynomial on the mesh rectangle $\tau$. Then
\begin{align}
\Vert \chi_{x} \Vert_{L^{p}(\tau)} \le Ch^{-1}_{x,\tau}\Vert \chi \Vert_{L^{p}(\tau)},
\quad
\Vert \chi_{y} \Vert_{L^{p}(\tau)} \le Ch^{-1}_{y,\tau}\Vert \chi \Vert_{L^{p}(\tau)},\label{eq:(3.3)}\\
\int^{y_{j}}_{y_{j-1}}|\chi(x_{i},y)|\mathrm{d}y
\le Ch^{-1}_{x,\tau} \Vert \chi \Vert_{L^{1}(\tau_{ij})},\label{eq:(3.4)}\\
\Vert \chi \Vert_{L^{q}(\tau)}
\le
C(h_{x,\tau}h_{y,\tau})^{1/q-1/p} \Vert \chi \Vert_{L^{p}(\tau)}
\quad \text{for} \quad p,q\in [1,\infty]\label{eq:(3.5)}.
\end{align}
\begin{lemma}\label{subdomain estimates}
Let Assumption $2.1$ hold true. Then there exists a constant $C$ such that the following interpolation error estimates hold true
\begin{subequations}\label{eq:(3.2)}\begin{align}
 &\Vert E-E^{I} \Vert_{L^{\infty}(\Omega_{s})}\le
 CN^{-\rho},&
 &\Vert E_{1}-E^{I}_{1} \Vert_{L^{\infty}(\Omega_{2})}
 \le  CN^{-\rho},\label{eq:(3.2a)}\\
&\Vert E_{12}-E^{I}_{12} \Vert_{L^{\infty}(\Omega_{2})}
 \le  CN^{-\rho},&
&\Vert (E_{12}-E^{I}_{12})_{x} \Vert_{L^{\infty}(\Omega_{1})}
\le  C\varepsilon^{-1}N^{-\rho},\label{eq:(3.2b)}\\
&\Vert (E_{1}-E^{I}_{1})_{y} \Vert_{L^{\infty}(\Omega_{2})}
\le  CN^{-\rho},&
&\Vert (E_{12}-E^{I}_{12})_{y} \Vert_{L^{\infty}(\Omega_{2})}
\le  C\varepsilon^{-1/2}N^{-\rho},\label{eq:(3.2c)}\\
&\Vert (E_{2}-E^{I}_{2})_{x} \Vert_{L^{\infty}(\Omega_{1})}
\le  CN^{-\rho},&
 &\Vert (E_{12}-E^{I}_{12})_{x} \Vert_{L^{1}(\Omega_{2})}
 \le  C\varepsilon^{1/2} N^{-\rho},\label{eq:(3.2d)}\\
&\Vert \nabla(u-u^{I}) \Vert_{L^{1}(\Omega_{s})}
\le  CN^{-1},&
&\Vert \nabla(u-u^{I}) \Vert_{L^{1}(\Omega_{1})}
\le  CN^{-1}\ln N\label{eq:(3.2e)}
\end{align}
where the function $E$ can be any one of $E_{1}$, $E_{2}$
or $E_{12}$.
\end{subequations}
\end{lemma}
\begin{proof}
At first, we will prove the second inequality of \eqref{eq:(3.2d)}. The proof of \eqref{eq:(3.2a)}--\eqref{eq:(3.2c)} and the first inequality of \eqref{eq:(3.2d)} is similar.

Each bilinear basis function $\phi^{i,j}(x,y)$ satisfies
$\phi^{i,j}(x,y)=\phi_{i}(x)\phi^{j}(y)$ where $\phi_{i}$ and $\phi^{j}$ are piecewise linear basis functions. If $x\in[x_{i-1},x_{i}]$,
\begin{equation*}
\phi_{i-1}(x)=\frac{x_{i}-x}{h_{x,\tau}},\quad
\phi_{i}(x)=\frac{x-x_{i-1}}{h_{x,\tau}}.
\end{equation*}
The functions $\phi^{j-1}(y)$ and $\phi^{j}(y)$ are defined similarly in $[y_{j-1},y_{j}]$. We define $z(x,y):=E_{12}(x,y)$ and
$w(x,y):=e^{-\beta(1-x)/\varepsilon}(e^{-y/\sqrt{\varepsilon}}+e^{-(1-y)/\sqrt{\varepsilon}})$.
\par
A direct calculation and (2.1e) give
 \begin{equation}\label{E3x L1(Omegay)}
 \Vert (E_{12})_{x}\Vert_{L^{1}(\Omega_{2})}
\le
C\varepsilon^{-1} \Vert w(x,y)\Vert_{L^{1}(\Omega_{2})}
\le
C\varepsilon^{1/2}N^{-\rho}.
 \end{equation}
For $(E^{I}_{12})_{x}$, we have
\begin{equation}\label{EI3 x analysis 1}
\begin{array}{l}
(E^{I}_{12})_{x}|_{\tau_{ij}}:=(z^{I})_{x}|_{\tau_{ij}}
=
\frac{z_{i,j-1}-z_{i-1,j-1}}{h_{x,\tau}}\varphi^{j-1}(y)
+\frac{z_{i,j}-z_{i-1,j}}{h_{x,\tau}}\varphi^{j}(y)\\
=h^{-1}_{x,\tau}\left(\varphi^{j-1}(y)\int^{x_{i}}_{x_{i-1}}z_{x}(x,y_{j-1})\mathrm{d}x
+\varphi^{j}(y)\int^{x_{i}}_{x_{i-1}}z_{x}(x,y_{j})\mathrm{d}x\right)\\
=h^{-1}_{x,\tau}\left(\int^{x_{i}}_{x_{i-1}}z_{x}(x,y_{j-1})\mathrm{d}x
+\varphi^{j}(y)\int^{x_{i}}_{x_{i-1}}\int^{y_{j}}_{y_{j-1}}z_{xy}(x,t)\mathrm{d}x\mathrm{d}t\right).
\end{array}
\end{equation}
Similarly, we have
\begin{equation}\label{EI3 x analysis 2}
\begin{array}{c}
(z^{I})_{x}|_{\tau_{ij}}=h^{-1}_{x,\tau}\left(\int^{x_{i}}_{x_{i-1}}z_{x}(x,y_{j})\mathrm{d}x
-\varphi^{j-1}(y)\iint_{\tau_{ij}}z_{xy}\mathrm{d}x\mathrm{d}t\right).
\end{array}
\end{equation}

By direct calculations, we can obtain
\begin{equation}\label{w(x,y)}
\begin{matrix}
w(x,y)\ge w(x,y_{j})>0&\text{if $[y_{j-1},y_{j}]\subset[0,\lambda_{y}]$},\\
w(x,y)\ge w(x,y_{j-1})>0&\text{if $[y_{j-1},y_{j}]\subset[1-\lambda_{y},1]$}
\end{matrix}
\end{equation}
for $y\in[y_{j-1},y_{j}]$. Then, for any $\tau_{ij}\in\Omega_{2}$ and $(x,y)\in \tau_{ij}$, from (2.1e) we have
\begin{equation}\label{E3xy tauij}
|z_{xy}(x,y)|\le C\varepsilon^{-3/2}w(x,y)
\end{equation}
and
\begin{align}
|z_{x}(x,y_{j})|\le C\varepsilon^{-1}w(x,y)\quad\text{if $[y_{j-1},y_{j}]\subset[0,\lambda_{y}]$}
\label{E3x Omega y down}\\
|z_{x}(x,y_{j-1})|\le C\varepsilon^{-1}w(x,y)\quad\text{if $[y_{j-1},y_{j}]\subset[1-\lambda_{y},1]$}
\label{E3x Omega y up}
\end{align}
where we have used \eqref{w(x,y)}. Combining
\eqref{EI3 x analysis 1},\eqref{E3xy tauij}, \eqref{E3x Omega y up} or
\eqref{EI3 x analysis 2}, \eqref{E3xy tauij}, \eqref{E3x Omega y down} and
considering $0\le \phi^{j-1}(y),\phi^{j}(y)\le 1$, we obtain
\begin{align*}
&\Vert (E^{I}_{12})_{x} \Vert_{L^{1}(\tau_{ij})}=\iint_{\tau_{ij}} |(z^{I})_{x}|\mathrm{d}x\mathrm{d}y\\
&\le
C\iint_{\tau_{ij}} h^{-1}_{x,\tau}
\left(
\int_{x_{i-1}}^{x_{i}}\varepsilon^{-1} w(x,y)\mathrm{d}x
+ \iint_{\tau_{ij}} \varepsilon^{-3/2}w(x,t)\mathrm{d}x\mathrm{d}t
\right)\mathrm{d}x\mathrm{d}y\\
&\le
Ch^{-1}_{x,\tau}(h_{x,\tau}\varepsilon^{-1}+h_{x,\tau}h_{y,\tau}\varepsilon^{-3/2})\Vert w \Vert_{L^{1}(\tau_{ij})}.
\end{align*}
Then, we have
\begin{align*}
&\Vert (E^{I}_{12})_{x}\Vert_{L^{1}(\Omega_{2})}
=
\sum_{\tau_{ij}\in\Omega_{2}}\Vert (E^{I}_{12})_{x}\Vert_{L^{1}(\tau_{ij})}\\
&\le
C(\varepsilon^{-1}+\varepsilon^{1/2}N^{-1}\ln N\cdot\varepsilon^{-3/2})\Vert w(x,y) \Vert_{L^{1}(\Omega_{2})}\\
&\le
C\varepsilon^{1/2}N^{-\rho}.
\end{align*}

From \eqref{eq:(2.1a)}, we have
\begin{equation*}
u-u^{I}=(S-S^{I})+(E_{1}-E^{I}_{1})+(E_{2}-E^{I}_{2})+(E_{12}-E^{I}_{12}).
\end{equation*}
  For the proof of the first inequality of \eqref{eq:(3.2e)},  from the standard interpolation theory and \eqref{eq:(2.1b)}, we have
  \begin{equation}\label{nabla(S-SI) L1(Omega y)}
   \Vert \nabla (S-S^{I}) \Vert_{L^{1}(\Omega_{s})}
  \le
  CN^{-1}\sum_{i+j=2} \left\Vert \frac{\partial^{i+j}S}{\partial x^{i}\partial y^{j}} \right\Vert_{L^{1}(\Omega_{s})}
  \le   CN^{-1}.
  \end{equation}
  \eqref{eq:(2.1c)} and the inverse estimates \eqref{eq:(3.4)} give
  \begin{align}
   \Vert \nabla (E_{1}-E^{I}_{1}) \Vert_{L^{1}(\Omega_{s})}
   &\le
   \Vert \nabla E_{1} \Vert_{L^{1}(\Omega_{s})}
   +\Vert \nabla E^{I}_{1} \Vert_{L^{1}(\Omega_{s})}
   \label{nabla(E1-EI1) L1(Omega s)}\\
   &\le
   \Vert \nabla E_{1} \Vert_{L^{1}(\Omega_{s})}
   +CN\Vert  E^{I}_{1} \Vert_{L^{1}(\Omega_{s})}
   \le
   CN^{1-\rho}\nonumber
  \end{align}
  where we have used $\left\vert E^{I}_{1}(x,y)\right\vert\le CN^{-\rho}$ for $(x,y)\in \Omega_{s}$. Similarly, we have
  \begin{equation}\label{nabla(E2-E2I) L1(Omega s) and E12}
  \Vert \nabla (E_{2}-E^{I}_{2}) \Vert_{L^{1}(\Omega_{s})}
  +\Vert \nabla (E_{12}-E^{I}_{12}) \Vert_{L^{1}(\Omega_{s})}
  \le
  CN^{1-\rho}.
  \end{equation}
  Combining \eqref{nabla(S-SI) L1(Omega y)}, \eqref{nabla(E1-EI1) L1(Omega s)},
  \eqref{nabla(E2-E2I) L1(Omega s) and E12}, we are done.
  \par
  For the estimate of the second inequality of \eqref{eq:(3.2e)},  on the one hand, we apply the standard bounds \cite[Lemma 3.1]{Linb1Styn2:2001-SDFEM} to
 $\nabla(S-S^{I})$, $\nabla(E_{1}-E^{I}_{1})$, $\nabla(E_{2}-E^{I}_{2})$ and
 $(E_{12}-E^{I}_{12})_{y}$ and on the other hand, apply the similar analytic techniques as for the
 second inequality of \eqref{eq:(3.2d)} to $(E_{12}-E^{I}_{12})_{x}$.
\end{proof}

\section{The discrete Green's function}

In this section, we will introduce the discrete Green's function and derive some estimates of it .
\par
Let $\boldsymbol{x}^{\ast}=(x^{\ast},y^{\ast})$ be a mesh node in $\Omega$. The discrete Green's function $G\in V^{N}$ associated with $\boldsymbol{x}^{\ast}$ is defined by
\begin{equation}\label{defition of Green}
B(v^{N},G)=v^{N}(\boldsymbol{x}^{\ast}),\,\forall v^{N}\in V^{N}.
\end{equation}

We introduce
\begin{equation}\label{sigma xy}
\sigma_{x}=kN^{-1}\ln N,\quad \sigma_{y}=kN^{-1/2}.
\end{equation}
The constants $k > 0$,
sufficiently large and independent of $N$ and $\varepsilon$, are chosen according to the derivation of
Theorem \ref{energyGreen}.
\begin{theorem}\label{energyGreen}
For $\boldsymbol{x}^{*}\in \Omega_{s}\cup\Omega_{1}$ we have
\begin{equation*}
\vert\vert\vert G \vert\vert\vert^{2}\le
8\vert\vert\vert G \vert\vert\vert^{2}_{\omega}\le
 CN\ln N.
\end{equation*}

\end {theorem}
\begin{proof}
See \cite[Theorem 4.1]{Zha1Mei2Chen3:submitted}
\end{proof}
For pointwise bounds on $G$ and its first-order derivatives, we define a subdomain of $\Omega$ as
\begin{equation*}
\Omega_{0}:=\{
\boldsymbol{x}\in\Omega :x-x^{*}\le K \sigma_{x}\ln N \text{ and }
\vert y-y^{*} \vert
\le K \sigma_{y}\ln N \}.
\end{equation*}
The constant $K>0$ will be chosen later.
\par
We extend $\Omega_{0}$ to the smallest mesh domain
$\Omega^{\prime}_{0}=\Omega^{\prime}_{0}(\boldsymbol{x}^{\ast})$, i.e.,
\begin{equation*}
\Omega^{\prime}_{0}=\cup\{
\tau\in\Omega^{N}:\mbox{meas}(\Omega_{0}\cap\tau)\ne0
\}.
\end{equation*}
Note that $\mbox{meas}(\Omega'_{0})\le C\sigma_{y}\ln N$.

\begin{theorem}\label{pointGreen}
Assume that $\sigma_{x}=kN^{-1}\ln N$ and $\sigma_{y}=kN^{-1/2}$, where $k>0$
is sufficiently large and independent of $\varepsilon$ and $N$. Let $\boldsymbol{x}^{*}\in\Omega_{s}\cup\Omega_{1}$, then
for each nonnegative integer $\upsilon$, there exists a positive constant $C=C(\upsilon)$ and $K=K(\upsilon)$ such that
\begin{align*}
&\Vert  G \Vert_{W^{1,\infty}(\Omega_{s}\setminus\Omega'_{0})} \le CN^{-\upsilon},\\
 \varepsilon\vert  G \vert_{W^{1,\infty}(\Omega_{1}\setminus\Omega'_{0})}+&
\Vert G \Vert_{L^{\infty}(\Omega_{1}\setminus\Omega'_{0})} \le CN^{-\upsilon},\\
 \varepsilon\vert  G \vert_{W^{1,\infty}(\Omega_{12}\setminus\Omega'_{0})}+&
\Vert G \Vert_{L^{\infty}(\Omega_{12}\setminus\Omega'_{0})} \le C\varepsilon^{-1/4}N^{-\upsilon}
\end{align*}
and
\begin{equation*}
 \varepsilon^{1/4}\Vert
 G_{x}\Vert_{L^{\infty}(\Omega_{2}\setminus\Omega'_{0})}+
 \varepsilon^{3/4}\Vert  G_{y}\Vert_{L^{\infty}(\Omega_{2}\setminus\Omega'_{0})}+
\varepsilon^{1/4}\Vert G \Vert_{L^{\infty}(\Omega_{2}\setminus\Omega'_{0})} \le CN^{-\upsilon}.
\end{equation*}
\end {theorem}
\begin{proof}
The proof is similar to the one of \cite[Theorem 4.2]{Linb1Styn2:2001-SDFEM}.
\end{proof}

\section{Maximum-Norm error estimates}
\noindent
In this section we shall derive bounds on $\vert (u-U)(\boldsymbol{x}) \vert$ for $\boldsymbol{x}$ lying in the various subregions of $\Omega$.
\par
 Taking $v^{N}=U-u^{I}$ in \eqref{defition of Green} yields
 \begin{equation*}
 (U-u)(\boldsymbol{x}^{\ast})=(U-u^{I})(\boldsymbol{x}^{\ast})=B(U-u^{I},G)
 \end{equation*}
 where $\boldsymbol{x}^{*}$ is a mesh node.
 Let $e$ denote the interpolation error, i.e.,
 \begin{equation*}
 e(\boldsymbol{x}):=(u-u^{I})(\boldsymbol{x}).
 \end{equation*}
 Then from \eqref{weak formulation},\eqref{SDFEM} and \eqref{bilinear form} we have
 \begin{equation}\label{estimate idea}
(U-u)(\boldsymbol{x}^{*})=-\varepsilon (\Delta u,\delta b G_{x})+B(e,G).
\end{equation}
The various terms on the right-hand side are bounded separately.
\par
For the following analysis, we derive some useful local estimates.
\begin{lemma}\label{local estimates}
If Assumption $2.1$ hold true, then there exists a constant $C$ such that
\begin{align}
 &\Vert (E-E^{I})_{x} \Vert_{L^{1}\left(\Omega_{s}\bigcap\Omega'_{0}\right)}\le
 CN^{-\rho}\sigma_{y}\ln N,\tag{5.2a}\label{eq:(5.2a)}\\
  &\Vert ((E_{1}+E_{12})-(E^{I}_{1}+E^{I}_{12}))_{y} \Vert_{L^{1}\left(\Omega_{s}\bigcap\Omega'_{0}\right)}\le
 CN^{-\rho}\sigma_{y}\ln N,\tag{5.2b}\label{eq:(5.2b)}\\
 &\Vert \Delta(E_{1}+E_{12}) \Vert_{L^{1}(\Omega_{s}\cap\Omega^{\prime}_{0})}
 \le
 C\varepsilon^{-1}N^{-\rho}\sigma_{y}\ln N\tag{5.2c}\label{eq:(5.2c)}
\end{align}
where $\sigma_{y}$ as in \eqref{sigma xy} and the function $E$ can be any one of $E_{1}$, $E_{2}$
or $E_{12}$.
\end{lemma}
\setcounter{equation}{2}
\begin{proof}
The proof of \eqref{eq:(5.2a)} and \eqref{eq:(5.2b)} is similar to \cite[Lemma 4.1]{Zhan1Mei2:submitted}.
Inequality \eqref{eq:(5.2c)} can be deduced directly by Assumption 2.1 and the definition of
$\Omega^{\prime}_{0}$.
\end{proof}

\begin{theorem}\label{estimate bilinear}
Assume that $\sigma_{x}=kN^{-1}\ln N$, $\sigma_{y}=kN^{-1/2}$ and $\varepsilon\le N^{-1}$. Then, for $\boldsymbol{x}^{*}\in \Omega_{s} \cup \Omega_{1}$, we have
\begin{equation*}
|B(e,G)|\le C(N^{-9/4}+\varepsilon^{1/4}N^{-2})(\ln^{3}N)\cdot \vert\vert\vert G \vert\vert\vert.
\end{equation*}
\end{theorem}
\begin{proof}
For the following analysis, we define
\begin{equation*}
\tilde{\Omega}:=\left((\Omega_{s}\cup\Omega_{1})\cap\Omega^{\prime}_{0}\right)\cup
(\Omega_{2}\cup\Omega_{12})
\end{equation*}
and modify the bilinear form by mean of integration by parts and the decomposition \eqref{eq:(2.1a)} as follows:
\begin{align*}
B(e,G)=&((\varepsilon+b^{2}\delta)e_{x},G_{x})
+\varepsilon(e_{y},G_{y})+(b(S-S^{I})_{x},G)\\
&+(b(E_{2}-E^{I}_{2})_{x},G)_{\Omega_{2}\cup\Omega_{12}}
+(b(E_{2}-E^{I}_{2}),Gn_{x})_{\partial(\Omega_{s}\cup\Omega_{1})}\\
&-(b(E-E^{I}),G_{x})_{\Omega_{s}\cup\Omega_{1}}
-(b(E_{1}-E^{I}_{1}),G_{x})_{\Omega_{2}\cup\Omega_{12}}\\
&-(b(E_{12}-E^{I}_{12}),G_{x})_{\Omega_{2}\cup\Omega_{12}}
+(ce,\delta bG_{x})+(ce,G)
\end{align*}
where $E=E_{1}+E_{2}+E_{12}$, $n_{x}$ is the $x$-axis coordinate of the outward normal vector of $\partial(\Omega_{s}\cup\Omega_{1})$.
Note that $(b(E_{2}-E^{I}_{2}),G n_{x})_{\partial(\Omega_{s}\cup\Omega_{1})}=0$
because $G = 0$ on $\partial\Omega$.
\par
The discussion of
$B(e,G)$ will be separated into three parts. In (a), we will analyze
$((\varepsilon+b^{2}\delta)e_{x},G_{x})$ and $\varepsilon(e_{y},G_{y})$. In (b),
$(b(S-S^{I})_{x},G)$ and $(b(E_{2}-E^{I}_{2})_{x},G)_{\Omega_{2}\cup\Omega_{12}}$ will be
discussed. In (c), we will analyze the residual terms of $B(e,G)$.\\*
(a) In this part, based on the boundary layer behavior of $u$, we discuss $u-u^{I}$ by means of
the following decomposition
\begin{equation*}
u-u^{I}=(S-S^{I})+(E_{1}-E^{I}_{1})+(E_{2}-E^{I}_{2})+(E_{12}-E^{I}_{12}).
\end{equation*}
\par
According to Lemma \ref{Lin identities}, we have
\begin{equation*}
 \left((S-S^{I})_{x},G_{x} \right)_{\tau}
=\int_{\tau}S_{xyy}J_{\tau}(y)\left(G_{x}-\frac{2}{3}(y-y_{\tau})G_{xy}\right)
\mathrm{d}x\mathrm{d}y.
\end{equation*}
Then
\begin{equation*}
 \left| \left((S-S^{I})_{x},G_{x} \right)_{\tau}\right|
 \le
 Ch^{2}_{y,\tau}\Vert S_{xyy} \Vert_{\tau} \Vert G_{x} \Vert_{\tau}
\end{equation*}
where we have used the inverse inequalities \eqref{eq:(3.3)}. Thus,
\begin{align}\label{bilinearestimates start}
 &\left| \left((\varepsilon+b^{2}\delta)(S-S^{I})_{x},G_{x} \right)_{\tilde{\Omega}}\right|\\
\le
C&N^{-2}
\Vert (\varepsilon+b^{2}\delta)^{1/2}S_{xyy} \Vert_{\tilde{\Omega}}
\Vert (\varepsilon+b^{2}\delta)^{1/2}G_{x} \Vert_{\Omega}\nonumber\\
\le
C&N^{-5/2} (\sigma_{y}\ln N)^{1/2}
\vert\vert\vert G \vert\vert\vert
\le
CN^{-11/4}(\ln^{1/2}N)\vert\vert\vert G \vert\vert\vert\nonumber
\end{align}
where we have used $\text{meas}(\tilde{\Omega})\le C\sigma_{y}\ln N+C\varepsilon^{1/2}\ln N\le
C\sigma_{y}\ln N$.

Similarly, we have
\begin{align}\label{Sy}
&\left|\left((\varepsilon+b^{2}\delta)(E_{1}-E^{I}_{1})_{x},G_{x}\right)
_{\Omega_{1}\cap\Omega^{\prime}_{0}}\right|
\le
CN^{-9/4}(\ln^{1/2}N) \vert\vert\vert G \vert\vert\vert,\\
&\left|\left((\varepsilon+b^{2}\delta)(E_{1}-E^{I}_{1})_{x},G_{x}\right)
_{\Omega_{2}\cup\Omega_{12}}\right|
\le
C\varepsilon^{3/4}N^{-3}(\ln^{5/2}N)
\vert\vert\vert G \vert\vert\vert,\\
&\left|\left( (\varepsilon+b^{2}\delta)(E_{2}-E^{I}_{2})_{x},G_{x} \right)_{\Omega_{2}\cup\Omega_{12}}\right|
\le
C\varepsilon^{1/4}N^{-5/2}(\ln^{2}N)
 \vert\vert\vert G \vert\vert\vert,\\
&\left|\left((\varepsilon+b^{2}\delta)(E_{12}-E^{I}_{12})_{x},G_{x}\right)_{\Omega_{12}}\right|
\le
C\varepsilon^{1/4}N^{-2}(\ln^{2}N) \vert\vert\vert G \vert\vert\vert,\\
&\varepsilon\left| \left((S-S^{I})_{y},G_{y} \right)_{\tilde{\Omega}}\right|
\le
C\varepsilon^{1/2}N^{-9/4}(\ln^{1/2}N)\vert\vert\vert G \vert\vert\vert,\\
&\varepsilon\left| \left((E_{2}-E^{I}_{2})_{y},G_{y} \right)
_{\Omega_{s}\cap\Omega^{\prime}_{0}}\right|
\le
C\varepsilon^{1/4}N^{-(2+\rho)}\vert\vert\vert G \vert\vert\vert,\\
&\varepsilon\left|\left( (E-E^{I})_{y}, G_{y} \right)_{\Omega_{1}\cap\Omega^{\prime}_{0}}\right|
\le
C\varepsilon^{3/4}N^{-2}(\ln^{2}N) \vert\vert\vert G \vert\vert\vert,\\
&\varepsilon\left|\left( (E-E^{I})_{y}, G_{y} \right)_{\Omega_{12}}\right|
\le
C\varepsilon^{3/4}N^{-2}(\ln^{2}N)\vert\vert\vert G \vert\vert\vert,\\
&\varepsilon\left|\left( (E_{2}-E^{I}_{2})_{y},G_{y} \right)_{\Omega_{2}}\right|
\le
C \varepsilon^{1/4}N^{-2}
\vert\vert\vert G \vert\vert\vert.
\end{align}

Furthermore Lemma \ref{local estimates} and the inverse inequality \eqref{eq:(3.5)} imply
\begin{align}\label{Ex-s}
&\left|((\varepsilon+b^{2}\delta)(E-E^{I})_{x},G_{x})_{\Omega_{s}\cap\Omega^{\prime}_{0}}\right|\\
\le
C&N^{-1}\Vert (E-E^{I})_{x} \Vert_{L^{1}(\Omega_{s}\cap\Omega^{\prime}_{0})}
\cdot \Vert G_{x} \Vert_{L^{\infty}(\Omega_{s}\cap\Omega^{\prime}_{0})}\nonumber\\
\le
C&N^{-1}N^{-\rho}\sigma_{y}(\ln N)\cdot N\Vert G_{x}\Vert_{\Omega_{s}\cap\Omega^{\prime}_{0}}\nonumber\\
\le
C&N^{-\rho}(\ln N)\vert\vert\vert G \vert\vert\vert.\nonumber
\end{align}
Similar argument shows
\begin{align}
&\varepsilon\left|\big((E_{1}-E^{I}_{1})_{y},G_{y}\big)_{\Omega_{s}\cap\Omega^{\prime}_{0}}\right|
\le
C\varepsilon^{1/2}N^{1/2-\rho}(\ln N)\vert\vert\vert G \vert\vert\vert,\\
&\varepsilon\left|\big((E_{12}-E^{I}_{12})_{y},G_{y}\big)_{\Omega_{s}\cap\Omega^{\prime}_{0}}\right|
\le
C\varepsilon^{1/2}N^{1/2-\rho}(\ln N)\vert\vert\vert G \vert\vert\vert,\\
&\left|\left((\varepsilon+b^{2}\delta) (E_{12}-E^{I}_{12})_{x},G_{x} \right)_{\Omega_{2}}\right|
\le
C\varepsilon^{1/4}N^{1/2-\rho}(\ln^{-1/2}N)
\vert\vert\vert G \vert\vert\vert
\end{align}
where we have used \eqref{eq:(3.2d)} in the last inequality.
\par
Lemma \ref{subdomain estimates} and  \selectlanguage{german}H"older\selectlanguage{english} inequalities give
\begin{align}
&\left|\left( (\varepsilon+b^{2}\delta)(E_{2}-E^{I}_{2})_{x},G_{x} \right)_{\Omega_{1}\cap\Omega^{\prime}_{0}}\right|\\
\le&
\varepsilon\Vert  (E_{2}-E^{I}_{2})_{x} \Vert_{L^{\infty}(\Omega_{1}\cap\Omega^{\prime}_{0})}
\Vert  G_{x} \Vert_{L^{1}(\Omega_{1}\cap\Omega^{\prime}_{0})}\nonumber\\
\le&
C\varepsilon N^{-\rho}(\varepsilon\sigma_{y}\ln^{2}N)^{1/2}
\Vert G_{x} \Vert_{\Omega_{1}\cap\Omega^{\prime}_{0}}\nonumber\\
\le&
C\varepsilon N^{-(1/4+\rho)}(\ln N)
\vert\vert\vert G \vert\vert\vert.\nonumber
\end{align}
Similarly, we have
\begin{align}
&\left|\left((\varepsilon+b^{2}\delta) (E_{12}-E^{I}_{12})_{x},G_{x} \right)_{\Omega_{1}\cap\Omega^{\prime}_{0}}\right|
\le
CN^{-(1/4+\rho)}(\ln N) \vert\vert\vert G \vert\vert\vert,\\
&\varepsilon\left|\left( (E_{1}-E^{I}_{1})_{y}, G_{y} \right)_{\Omega_{2}}\right|
\le
C\varepsilon^{3/4} N^{-\rho}(\ln^{1/2}N)
\vert\vert\vert G \vert\vert\vert,\\
&\varepsilon\left|\left( (E_{12}-E^{I}_{12})_{y},G_{y} \right)_{\Omega_{2}}\right|
\le
C \varepsilon^{1/4}N^{-\rho}(\ln^{1/2}N)
\vert\vert\vert G \vert\vert\vert.
\end{align}

In view of Theorem \ref{pointGreen} with $\upsilon=2$ and Lemma \ref{subdomain estimates}, we see
\begin{align}\label{secondorder end}
&\left|((\varepsilon+b^{2}\delta)(u-u^{I})_{x},G_{x})_{(\Omega_{s}\cup\Omega_{1})\setminus\Omega^{\prime}_{0}}
+\varepsilon((u-u^{I})_{y},G_{y})_{(\Omega_{s}\cup\Omega_{1})\setminus\Omega^{\prime}_{0}}\right|\\
\le&
C \Vert \nabla(u-u^{I}) \Vert_{L^{1}(\Omega_{s}\cup\Omega_{1})} (N^{-1}\Vert \nabla G \Vert_{L^{\infty}(\Omega_{s}\setminus\Omega^{\prime}_{0})}
+\varepsilon \Vert \nabla G \Vert_{L^{\infty}(\Omega_{1}\setminus\Omega^{\prime}_{0})})
\nonumber\\
\le&
CN^{-2}.\nonumber
\end{align}
(b) We see from Lemma \ref{Lin identities} that
\begin{align*}
\left(
(S-S^{I})_{x},G\right)
&=\sum_{\tau\in\Omega}\int_{\tau}R(S,G)\mathrm{d}x\mathrm{d}y+\sum_{\tau\in\Omega}\frac{h^{2}_{x,\tau}}{12}
\left(\int_{l_{2}}-\int_{l_{4}}\right)S_{xx}G\mathrm{d}y
\end{align*}
where $R(\cdot,\cdot)$ as in Lemma \ref{Lin identities}. Based on our Shishkin mesh and the properties of the discrete Green function $G$, we decompose the first term as follows:
\begin{align*}
&\sum_{\tau\in\Omega}
\int_{\tau}R(S,G)\mathrm{d}x\mathrm{d}y
=
\big(\sum_{\tau\in\tilde{\Omega}}
+\sum_{\tau\in(\Omega_{s}\cup\Omega_{1})\setminus\Omega^{\prime}_{0}}
\big)
\int_{\tau}R_{1}(S,G)+R_{2}(S,G)\mathrm{d}x\mathrm{d}y
\end{align*}
where
\begin{equation*}
R_{1}(S,G)=\frac{1}{3}F_{\tau}(x)(x-x_{\tau})S_{xxx}G_{x}-\frac{h^{2}_{x,\tau}}{12}S_{xxx}G
\end{equation*}
and
\begin{equation*}
R_{2}(S,G)=J_{\tau}(y)S_{xyy}\big(
G-(x-x_{\tau})G_{x}-\frac{2}{3}(y-y_{\tau})G_{y}+\frac{2}{3}(x-x_{\tau})
(y-y_{\tau})G_{xy}\big).
\end{equation*}

Firstly, from Assumption \ref{Assumption 1} and the definition of $\tilde{\Omega}$, we have
\begin{align}
&\sum_{\tau\in\tilde{\Omega}}\label{Sx G 1}
\int_{\tau}\frac{1}{3}|F_{\tau}(x)(x-x_{\tau})S_{xxx}G_{x}|\mathrm{d}x\mathrm{d}y\\
\le
C&\sum_{\tau\in\tilde{\Omega}}
h^{3}_{x,\tau}\Vert S_{xxx} \Vert_{L^{\infty}(\tau)}
\Vert G_{x} \Vert_{L^{1}(\tau)}\nonumber\\
\le
C&H^{3}_{x}\Vert G_{x} \Vert_{L^{1}((\Omega_{s}\cap\Omega^{\prime}_{0})\cup\Omega_{2})}
+Ch^{3}_{x}\Vert G_{x} \Vert_{L^{1}((\Omega_{1}\cap\Omega^{\prime}_{0})\cup\Omega_{12})}\nonumber\\
\le
C&N^{-11/4}(\ln^{1/2}N) \vert\vert\vert G\vert\vert\vert\nonumber
\end{align}
and
\begin{equation}\label{Sx G 2}
\sum_{\tau\in\tilde{\Omega}}
\int_{\tau}\frac{h^{2}_{x,\tau}}{12}|S_{xxx}G|\mathrm{d}x\mathrm{d}y
\le
CN^{-9/4}(\ln^{1/2}N) \vert\vert\vert G\vert\vert\vert.
\end{equation}
Applying inverse inequalities \eqref{eq:(3.3)} to the last part of $R_{2}(S,G)$, we obtain
\begin{align}\label{Sx G 3}
\sum_{\tau\in\tilde{\Omega}}
\int_{\tau}
|R_{2}(S,G)|\mathrm{d}x\mathrm{d}y
&\le
C\sum_{\tau\in\tilde{\Omega}}
h^{2}_{y,\tau}\Vert S_{xyy} \Vert_{\tau}
\Vert G \Vert_{\tau}\\
&\le
CN^{-2}\cdot (\sigma_{y}\ln N)^{1/2}\cdot \vert\vert\vert G\vert\vert\vert\nonumber\\
&\le
CN^{-9/4}(\ln^{1/2}N) \vert\vert\vert G\vert\vert\vert.\nonumber
\end{align}
From Theorem \ref{pointGreen} with $\upsilon=1$, we have
\begin{align}\label{Sx G 4}
&\sum_{\tau\in(\Omega_{s}\cup\Omega_{1})\setminus\Omega^{\prime}_{0}}
\int_{\tau}|R_{1}(S,G)|\mathrm{d}x\mathrm{d}y\\
\le
C&\sum_{\tau\in(\Omega_{s}\cup\Omega_{1})\setminus\Omega^{\prime}_{0}}
\left(h^{3}_{x,\tau}\Vert S_{xxx} \Vert_{L^{1}(\tau)}
\Vert G_{x} \Vert_{L^{\infty}(\tau)}+h^{2}_{x,\tau}\Vert S_{xxx} \Vert_{L^{1}(\tau)}
\Vert G \Vert_{L^{\infty}(\tau)}\right)\nonumber\\
\le
C&H^{3}_{x}\Vert G_{x} \Vert_{L^{\infty}(\Omega_{s}\setminus\Omega^{\prime}_{0})}
+Ch^{3}_{x}\Vert G_{x} \Vert_{L^{\infty}(\Omega_{1}\setminus\Omega^{\prime}_{0})}
+CN^{-2}\Vert G \Vert
_{L^{\infty}((\Omega_{s}\cup\Omega_{1})\setminus\Omega^{\prime}_{0})}\nonumber\\
\le
C&N^{-2}\nonumber
\end{align}
and
\begin{align}\label{Sx G 5}
\sum_{\tau\in(\Omega_{s}\cup\Omega_{1})\setminus\Omega^{\prime}_{0}}
\int_{\tau}|R_{2}(S,G)|
&\le
C\sum_{\tau\in(\Omega_{s}\cup\Omega_{1})\setminus\Omega^{\prime}_{0}}
H^{2}_{y}\Vert S_{xyy} \Vert_{L^{1}(\tau)}\Vert G \Vert_{L^{\infty}(\tau)}\\
&\le
CN^{-2}\nonumber
\end{align}
where we have used inverse inequalities \eqref{eq:(3.3)}.

Secondly, we set $L:=\{(1-\lambda_{x},y):0\le y \le 1\}$.
Then we have
\begin{align*}
&\left|\sum_{\tau\in\Omega}\frac{h^{2}_{x,\tau}}{12}
\left(\int_{l_{2}}-\int_{l_{4}}\right)S_{xx}G\mathrm{d}y\right|
=\frac{1}{12}\left|\sum_{l\in L}(H^{2}_{x}-h^{2}_{x})\int_{l}S_{xx}G\mathrm{d}y \right|\\
\le&
CH^{2}_{x}\sum_{l\in L}\int_{l}\left|S_{xx}G\right|\mathrm{d}y
\le
CH^{2}_{x} \left(\sum_{l\in L\cap\Omega^{\prime}_{0}}\int_{l}\left|S_{xx}G\right|\mathrm{d}y+\sum_{l\in L\setminus\Omega^{\prime}_{0}} \int_{l}\left|S_{xx}G\right|\mathrm{d}y\right)\\
=&
CH^{2}_{x}(\mathrm{I}+\mathrm{II}).
\end{align*}
The estimate of I is straightforward:
\begin{align*}
\mathrm{I}&=\sum_{l\in L\cap\Omega^{\prime}_{0}}\int_{l}\left|S_{xx}G\right|(1-\lambda_{x},y)\mathrm{d}y
\le
\int^{y^{\prime}_{2}}_{y^{\prime}_{1}}\left|
\int^{1}_{1-\lambda_{x}}(S_{xxx}G+S_{xx}G_{x})\mathrm{d}x
\right|\mathrm{d}y\\
&\le
\Vert S_{xxx} \Vert_{L^{\infty}(D_{L})}\Vert G \Vert_{L^{1}(D_{L})}
+\Vert S_{xx} \Vert_{L^{\infty}(D_{L})}\Vert G_{x} \Vert_{L^{1}(D_{L})}\\
&\le
C(\varepsilon\sigma_{y}\ln^{2}N)^{1/2}
(\Vert G \Vert_{D_{L}}+\Vert G_{x} \Vert_{D_{L}})\\
&\le
CN^{-1/4}\ln N\cdot (\varepsilon^{1/2}\Vert G \Vert_{D_{L}}+\varepsilon^{1/2}\Vert G_{x} \Vert_{D_{L}})\\
&\le
CN^{-1/4}(\ln N)\Vert\vert G \vert\Vert
\end{align*}
where $\{1-\lambda_{x}\}\times[y^{\prime}_{1},y^{\prime}_{2}]=L\cap\Omega^{\prime}_{0}$ and
$D_{L}:=[1-\lambda_{x},1]\times[y^{\prime}_{1},y^{\prime}_{2}]$.\\*
From Theorem \ref{pointGreen} with $\upsilon=1$, we have
\begin{align*}
\mathrm{II}
&=
\sum_{l\in L\cap(\Omega_{s}\setminus\Omega^{\prime}_{0})}
\int_{l}\left|S_{xx}G\right|\mathrm{d}y
+
\sum_{l\in L\cap(\Omega_{2}\setminus\Omega^{\prime}_{0})}
\int_{l}\left|S_{xx}G\right|\mathrm{d}y\\
&\le
C\Vert G \Vert_{L^{\infty}(\Omega_{s}\setminus\Omega^{\prime}_{0})}
+C\varepsilon^{1/2}(\ln N)\Vert G \Vert_{L^{\infty}(\Omega_{2}\setminus\Omega^{\prime}_{0})}\\
&\le
CN^{-1}.
\end{align*}
Considering the estimates for I and II, we obtain
\begin{equation}
\left|\sum_{\tau\in\Omega}\frac{h^{2}_{x,\tau}}{12}
\left(\int_{l_{2}}-\int_{l_{4}}\right)S_{xx}G\mathrm{d}y\right|
\le
CN^{-9/4}(\ln N)\vert\vert\vert G \vert\vert\vert.
\end{equation}

The estimates for $((E_{2}-E^{I}_{2})_{x},G)_{\Omega_{2}\cup\Omega_{12}}$ are the same as $((S-S^{I})_{x},G)$. Thus we have
\begin{equation}
|b((E_{2}-E^{I}_{2})_{x},G)_{\Omega_{2}\cup\Omega_{12}}|
\le
C\varepsilon^{1/4}N^{-2}(\ln^{2} N) \vert\vert\vert G \vert\vert\vert.
\end{equation}
(c) From Lemma \ref{subdomain estimates}, we have
\begin{align}\label{3-E_all-s}
 \left|
b\left(E-E^{I},G_{x}\right)_{\Omega_{s}\cap\Omega^{\prime}_{0}}\right|
&\le
C\Vert E-E^{I} \Vert_{L^{\infty}(\Omega_{s}\cap\Omega^{\prime}_{0})}
\Vert G_{x} \Vert_{L^{1}(\Omega_{s}\cap\Omega^{\prime}_{0})}\\
&\le
CN^{-\rho}\cdot(\sigma_{y}\ln N)^{1/2} \Vert G_{x} \Vert_{\Omega_{s}\cap\Omega^{\prime}_{0}}\nonumber\\
&\le
CN^{1/4-\rho}(\ln^{1/2} N)
\vert\vert\vert G \vert\vert\vert.\nonumber
\end{align}
Similarly, we have
\begin{align}
&\left|b\left(E_{1}-E^{I}_{1},G_{x}\right)_{\Omega_{2}}\right|
\le
C\varepsilon^{1/4}N^{1/2-\rho}(\ln^{1/2}N)\vert\vert\vert G \vert\vert\vert,\\
&\left|b\left(E_{12}-E^{I}_{12},G_{x}\right)_{\Omega_{2}}\right|
\le
C\varepsilon^{1/4}N^{1/2-\rho}(\ln^{1/2}N)\vert\vert\vert G \vert\vert\vert.
\end{align}

In view of Lemma \ref{interpolation}, we obtain
\begin{align}
 \left|b
\left(E-E^{I},G_{x}\right)_{\Omega_{1}\cap\Omega^{\prime}_{0}}\right|
&\le
C\Vert E-E^{I} \Vert_{L^{\infty}(\Omega_{1}\cap\Omega^{\prime}_{0})}
\Vert G_{x} \Vert_{L^{1}(\Omega_{1}\cap\Omega^{\prime}_{0})}\\
&\le
CN^{-2}(\ln^{2}N)\cdot(\varepsilon\sigma_{y}\ln^{2}N)^{1/2} \Vert G_{x} \Vert_{\Omega_{1}\cap\Omega^{\prime}_{0}}\nonumber\\
&\le
CN^{-9/4}(\ln^{3}N)
\vert\vert\vert G \vert\vert\vert.\nonumber
\end{align}
Similar argument shows
\begin{align}
&\left|b\left(E_{1}-E^{I}_{1},G_{x}\right)_{\Omega_{12}}\right|
\le
C\varepsilon^{1/4}N^{-2}(\ln^{3}N)\cdot\vert\vert\vert G \vert\vert\vert,\\
&\left|b\left(E_{12}-E^{I}_{12},G_{x}\right)_{\Omega_{12}}\right|
\le
C\varepsilon^{1/4}N^{-2}(\ln^{3}N)\cdot\vert\vert\vert G \vert\vert\vert,\\
&\left|c(u-u^{I},\delta G_{x})\right|
\le
CN^{-5/2}(\ln^{2}N)\vert\vert\vert G \vert\vert\vert,\\
&\left|c(u-u^{I},G)_{\tilde{\Omega}}\right|
\le
C(N^{-9/4}\ln^{1/2} N+\varepsilon^{1/4}N^{-2}\ln^{5/2}N) \vert\vert\vert G \vert\vert\vert.
\end{align}
Theorem \ref{pointGreen} with $\upsilon=1$ and Lemma \ref{interpolation} yield
\begin{align}
&\left|b
\left(E-E^{I},G_{x}\right)_{\Omega_{1}\setminus\Omega^{\prime}_{0}}\right|\\
\le C&
\Vert E-E^{I} \Vert_{L^{\infty}(\Omega_{1})}
\mathrm{meas}(\Omega_{1}\setminus\Omega^{\prime}_{0})
\Vert G_{x}\Vert_{L^{\infty}(\Omega_{1}\setminus\Omega^{\prime}_{0})}\nonumber\\
\le C&
N^{-2}\ln^{2}N(\varepsilon \ln N)
\Vert G_{x}\Vert_{L^{\infty}(\Omega_{1}\setminus\Omega^{\prime}_{0})}
\le CN^{-2}\nonumber
\end{align}
and
\begin{align}
&\left|b
\left(E-E^{I},G_{x}\right)_{\Omega_{s}\setminus\Omega^{\prime}_{0}}\right|
\le C N^{-2},\\
&\left|(u-u^{I},G)_{(\Omega_{s}\cup\Omega_{1})\setminus\Omega^{\prime}_{0}}\right|
\le
CN^{-2}\label{bilinearestimates end}.
\end{align}

Collecting \eqref{bilinearestimates start}--\eqref{bilinearestimates end}, we are done.

\end{proof}

\begin{theorem}\label{estimate secondorder}
Assume that $u$ satisfies Assumption \ref{Assumption 1} and $\varepsilon\le N^{-1}$. Let $\sigma_{x}=kN^{-1}\ln N$ and $\sigma_{y}=kN^{-1/2}$.
Then for any mesh node $\boldsymbol{x}^{\ast}\in\Omega_{s}\cup\Omega_{1}$
\begin{enumerate}
\item if $\Omega^{\prime}_{0}\subset \Omega_{s}\cup\Omega_{1}$, we have
\begin{equation*}
\left|\left(\varepsilon\Delta u,\delta b G_{x}\right)
\right|\le
C(N^{-\rho}+\varepsilon N^{-5/4})(\ln N) \vert\vert\vert G \vert\vert\vert.
\end{equation*}
\item if $\Omega^{\prime}_{0}\not\subset \Omega_{s}\cup\Omega_{1}$, we have
\begin{equation*}
\left|\left(\varepsilon\Delta u,\delta b G_{x}\right)
\right|\le
C(N^{-\rho}+\varepsilon^{1/4}\delta_{y})(\ln N) \vert\vert\vert G \vert\vert\vert.
\end{equation*}
\end{enumerate}
\end{theorem}
\begin{proof}
 We set $E=E_{1}+E_{2}+E_{12}$ and  define
$\Gamma_{s,x}:=\Omega_{s}\cap\Omega_{1}$
and $\Gamma_{y,xy}:=\Omega_{2}\cap\Omega_{12}$. At the beginning, integration by parts and the definition of $\delta$ yield
\begin{align*}
(\Delta S,\delta bG_{x})=&
(\Delta S,\delta bG_{x})_{\Omega_{s}\cup\Omega_{2}}\\
=&(\Delta S,\delta_{s} bG)_{\Gamma_{s,x}}+
(\Delta S,\delta_{y} bG)_{\Gamma_{y,xy}}
-((\Delta S)_{x},\delta bG)
\end{align*}
and
\begin{equation*}
(\Delta E_{2},\delta bG_{x})=(\Delta E_{2},\delta_{s} bG)_{\Gamma_{s,x}}+
(\Delta E_{2},\delta_{y} bG)_{\Gamma_{y,xy}}
-((\Delta E_{2})_{x},\delta bG).
\end{equation*}
Thus,
\begin{align*}
(\varepsilon\Delta u,\delta bG_{x})
=&\varepsilon\delta_{s}(\Delta (S+E_{2}), bG)_{\Gamma_{s,x}}+
\varepsilon\delta_{y}(\Delta (S+E_{2}), bG)_{\Gamma_{y,xy}}\\
&-\varepsilon((\Delta (S+E_{2}))_{x},\delta bG
)+\varepsilon\left(\Delta (E_{1}+E_{12}),\delta bG_{x}\right).
\end{align*}
The terms on the right-hand side are analyzed separately.
\par
From Assumption \ref{Assumption 1}, we have
\begin{align}\label{Gamma s x S}
&\varepsilon\delta_{s} \left|(\Delta S, bG)_{\Gamma_{s,x}}\right|
\le
C\varepsilon N^{-1}\Vert\Delta S\Vert_{L^{\infty}(\Gamma_{s,x})}
\Vert G \Vert_{L^{1}(\Gamma_{s,x})}\\
\le
C&\varepsilon N^{-1}\left(
\int_{\Gamma_{s,x}\cap\Omega^{\prime}_{0}}|G(1-\lambda_{x},y)|\mathrm{d}y
+\Vert G \Vert_{L^{1}(\Gamma_{s,x}\setminus\Omega^{\prime}_{0})}
\right)\nonumber\\
\le
C&\varepsilon N^{-1}\left(
\int_{\Gamma_{s,x}\cap\Omega^{\prime}_{0}}\int^{1}_{1-\lambda_{x}}|G_{x}(x,y)|\mathrm{d}x\mathrm{d}y
+\Vert G \Vert_{L^{\infty}(\Omega_{s}\setminus\Omega^{\prime}_{0})}
\right)\nonumber\\
\le
C&\varepsilon N^{-1}(\varepsilon\sigma_{y}\ln^{2}N)^{1/2} \Vert G_{x} \Vert_{\Omega_{1}}
+C\varepsilon N^{-1}\Vert G \Vert_{L^{\infty}(\Omega_{s}\setminus\Omega^{\prime}_{0})}\nonumber\\
\le
C&\varepsilon N^{-5/4}(\ln N) \vert\vert\vert G \vert\vert\vert+C\varepsilon N^{-2}\nonumber
\end{align}
where we have used Theorem \ref{pointGreen} with $\upsilon=1$.
Similarly, we have
\begin{align}
&\varepsilon\delta_{s} \left|(\Delta E_{2}, bG)_{\Gamma_{s,x}}\right|
\le
CN^{-(1+\rho)}(\ln^{1/2}N) \vert\vert\vert G \vert\vert\vert,\label{Gamma s x E2}\\
&\varepsilon\delta_{y} \left|(\Delta (S+ E_{2}), bG)_{\Gamma_{y,xy}}\right|
\le
C\varepsilon^{1/4} \delta_{y}(\ln N)
\vert\vert\vert G \vert\vert\vert\label{Gamma y xy}.
\end{align}
\par
Considering \eqref{eq:(2.1b)} and $\mbox{meas}(\Omega\cap\Omega^{\prime}_{0})\le C\sigma_{y}\ln N$, we obtain
 \begin{align}\label{secondorder S G}
 \left|
 \varepsilon\left((\Delta S)_{x},\delta bG\right)_{\Omega\cap\Omega^{\prime}_{0}}
 \right|
 &\le
 C\varepsilon N^{-1}\Vert (\Delta S)_{x} \Vert_{L^{\infty}(\Omega\cap\Omega^{\prime}_{0})}
\Vert G \Vert_{L^{1}(\Omega\cap\Omega^{\prime}_{0})}\\
 &\le
 C\varepsilon N^{-1}(\sigma_{y}\ln N)^{1/2}
 \Vert G \Vert_{\Omega\cap\Omega^{\prime}_{0}}\nonumber\\
&\le
C\varepsilon N^{-5/4}(\ln^{1/2} N) \vert\vert\vert G \vert\vert\vert.\nonumber
 \end{align}
Assumption \ref{Assumption 1} and the inverse inequality \eqref{eq:(3.5)} yield
\begin{align}\label{secondorder E2 G Omegas}
\left|\varepsilon((\Delta E_{2})_{x},\delta G)
_{\Omega_{s}\cap\Omega^{\prime}_{0}}\right|
&\le
 C\varepsilon \delta_{s}
 \Vert (\Delta E_{2})_{x} \Vert_{L^{1}(\Omega_{s}\cap\Omega^{\prime}_{0})}
 \Vert G \Vert_{L^{\infty}(\Omega_{s}\cap\Omega^{\prime}_{0})}\\
 &\le
 C\varepsilon \delta_{s} \cdot\varepsilon^{-1/2}N^{-\rho}\cdot (H_{x}H_{y})^{-1/2}
 \Vert G \Vert_{\Omega_{s}\cap\Omega^{\prime}_{0}}\nonumber\\
&\le
\varepsilon^{1/2}N^{-\rho}
\cdot\vert\vert\vert G \vert\vert\vert.\nonumber
 \end{align}
In view of \eqref{eq:(2.1d)}, we get
\begin{align}\label{secondorder E2 G Omegay}
\left|\varepsilon\delta_{y}((\Delta E_{2})_{x},G)_{\Omega_{2}\cap\Omega^{\prime}_{0}}\right|
&\le
C\varepsilon\delta_{y}\Vert (\Delta E_{2})_{x}\Vert_{\Omega_{2}}
\Vert G \Vert_{\Omega_{2}}\\
&\le
C\varepsilon^{1/4}\delta_{y}\cdot \vert\vert\vert G \vert\vert\vert.\nonumber
\end{align}
Lemma \ref{local estimates}, Assumption \ref{Assumption 1} and the inverse inequality \eqref{eq:(3.5)} yield
\begin{align}\label{secondorder E1E3 S}
 &\left|
 \varepsilon\left(\Delta (E_{1}+E_{12}),\delta bG_{x}\right)_{\Omega_{s}\cap\Omega^{\prime}_{0}}
 \right|\\
\le &
 C\varepsilon \delta_{s}
 \Vert \Delta (E_{1}+E_{12}) \Vert_{L^{1}(\Omega_{s}\cap\Omega^{\prime}_{0})}
 \Vert G_{x} \Vert_{L^{\infty}(\Omega_{s}\cap\Omega^{\prime}_{0})}\nonumber\\
\le &
 C\varepsilon \delta_{s} \cdot\varepsilon^{-1}N^{-\rho}\sigma_{y}\ln N\cdot (H_{x}H_{y})^{-1/2}
 \Vert G_{x} \Vert_{\Omega_{s}\cap\Omega^{\prime}_{0}}\nonumber\\
\le&
CN^{-\rho}(\ln N)\vert\vert\vert G \vert\vert\vert\nonumber
 \end{align}
and
 \begin{align}
 &\left|
 \varepsilon\left(\Delta (E_{1}+E_{12}),\delta bG_{x}\right)_{\Omega_{2}\cap\Omega^{\prime}_{0}} \right|\\
\le &
 C\varepsilon \delta_{y} \Vert \Delta (E_{1}+E_{12}) \Vert_{L^{1}(\Omega_{2}\cap\Omega^{\prime}_{0})}
 \Vert G_{x} \Vert_{L^{\infty}(\Omega_{2}\cap\Omega^{\prime}_{0})}\nonumber\\
\le &
 C\varepsilon \delta_{y} \cdot\varepsilon^{-1/2} N^{-\rho} (\ln N)\cdot (H_{x}h_{y})^{-1/2}
 \Vert G_{x} \Vert_{\Omega_{2}\cap\Omega^{\prime}_{0}}\nonumber\\
\le&
C\varepsilon^{1/4}N^{1/2-\rho}(\ln^{1/2} N)
\vert\vert\vert G \vert\vert\vert.\nonumber
 \end{align}
 \par
Considering \eqref{eq:(2.1b)}---\eqref{eq:(2.1e)} and Theorem \ref{pointGreen} with $\upsilon=1$, we obtain
 \begin{align}\label{Omega c}
 \left|
 \varepsilon\left((\Delta S)_{x},\delta bG\right)_{\Omega\setminus\Omega^{\prime}_{0}}
 \right|
 &\le
 C\varepsilon\delta
 \Vert (\Delta S)_{x} \Vert_{L^{1}(\Omega\setminus\Omega^{\prime}_{0})}
 \Vert G \Vert_{L^{\infty}(\Omega\setminus\Omega^{\prime}_{0})}\\
 &\le
 CN^{-2}.\nonumber
 \end{align}
Similar argument shows
\begin{align}
 \left|\varepsilon((\Delta E_{2})_{x},\delta G)
_{\Omega\setminus\Omega^{\prime}_{0}}\right|
&\le CN^{-2},\\
\left| \varepsilon\left(\Delta (E_{1}+E_{12}),\delta bG_{x}\right)_{\Omega\setminus\Omega^{\prime}_{0}} \right|
&\le CN^{-2}\label{secondorder E1E3 G end}.
 \end{align}
We deduce the last inequality by analyzing separately for
$\Omega_{s}\setminus\Omega^{\prime}_{0}$
and $\Omega_{2}\setminus\Omega^{\prime}_{0}$.
\par
If $\Omega^{\prime}_{0}\subset\Omega_{s}\cup\Omega_{1}$, we can get $\Omega_{2}\cap\Omega^{\prime}_{0}=\varnothing$. Then
\begin{align}
\left|\varepsilon\delta_{y}((\Delta E_{2})_{x},G)_{\Omega_{2}\cap\Omega^{\prime}_{0}}\right|=0,\\
\left|
 \varepsilon\left(\Delta (E_{1}+E_{12}),\delta bG_{x}\right)_{\Omega_{2}\cap\Omega^{\prime}_{0}} \right|=0.
\end{align}
In this case, the modification of \eqref{Gamma y xy} is that
\begin{align}\label{Gamma y xy modified}
&\varepsilon\delta_{y} (|\Delta S|+|\Delta E_{2}|, |bG|
)_{\Gamma_{y,xy}}\\
\le&
C\varepsilon\delta_{y} \Vert \Delta(S+E_{2})\Vert_{L^{1}(\Gamma_{y,xy})}
\Vert G \Vert_{L^{\infty}(\Gamma_{y,xy})}\nonumber\\
\le&
C\varepsilon^{1/2}\delta_{y}
\Vert G \Vert_{L^{\infty}(\Omega_{2}\setminus\Omega^{\prime}_{0})}\nonumber\\
\le&
C N^{-2}\nonumber
\end{align}
where we have used Theorem \ref{pointGreen} with $\upsilon=1$.
Collecting \eqref{Gamma s x S}, \eqref{Gamma s x E2}, \eqref{secondorder S G}, \eqref{secondorder E2 G Omegas}, \eqref{secondorder E1E3 S} and \eqref{Omega c}--\eqref{Gamma y xy modified}, we have
\begin{equation*}
\left|\left(\varepsilon\Delta u,\delta b G_{x}\right)
\right|\le
C(N^{-\rho}+\varepsilon N^{-5/4})(\ln N) \vert\vert\vert G \vert\vert\vert.
\end{equation*}
If $\Omega^{\prime}_{0}\not\subset\Omega_{s}\cup\Omega_{1}$,
using \eqref{Gamma s x S}--\eqref{secondorder E1E3 G end}
leads to
\begin{equation*}
\left|\left(\varepsilon\Delta u,\delta b G_{x}\right)
\right|\le
C(N^{-\rho}+\varepsilon^{1/4}\delta_{y})(\ln N) \vert\vert\vert G \vert\vert\vert.
\end{equation*}
\end{proof}

\bibliographystyle{plain}
\bibliography{alex-FE-paper,alex-FE-book}
\end{document}